\documentclass[british,english]{article}
\usepackage[latin9]{inputenc}
\usepackage{mathrsfs}
\usepackage{amsthm}
\usepackage{amsmath}
\usepackage{amssymb}
\usepackage{setspace}
\usepackage{esint}
\usepackage[authoryear]{natbib}

\makeatletter
\theoremstyle{plain}
\newtheorem{thm}{\protect\theoremname}
  \theoremstyle{plain}
  \newtheorem{prop}[thm]{\protect\propositionname}
  \theoremstyle{definition}
  \newtheorem*{example*}{\protect\examplename}
  \theoremstyle{plain}
  \newtheorem{assumption}[thm]{\protect\assumptionname}
  \theoremstyle{plain}
  \newtheorem{cor}[thm]{\protect\corollaryname}

\usepackage{bbm}
\usepackage{a4}

\renewcommand{\hat}{\widehat}
\renewcommand{\tilde}{\widetilde}

\makeatother

\usepackage{babel}
  \addto\captionsbritish{\renewcommand{\assumptionname}{Assumption}}
  \addto\captionsbritish{\renewcommand{\corollaryname}{Corollary}}
  \addto\captionsbritish{\renewcommand{\examplename}{Example}}
  \addto\captionsbritish{\renewcommand{\propositionname}{Proposition}}
  \addto\captionsbritish{\renewcommand{\theoremname}{Theorem}}
  \addto\captionsenglish{\renewcommand{\assumptionname}{Assumption}}
  \addto\captionsenglish{\renewcommand{\corollaryname}{Corollary}}
  \addto\captionsenglish{\renewcommand{\examplename}{Example}}
  \addto\captionsenglish{\renewcommand{\propositionname}{Proposition}}
  \addto\captionsenglish{\renewcommand{\theoremname}{Theorem}}
  \providecommand{\assumptionname}{Assumption}
  \providecommand{\corollaryname}{Corollary}
  \providecommand{\examplename}{Example}
  \providecommand{\propositionname}{Proposition}
\providecommand{\theoremname}{Theorem}

\begin{document}
\selectlanguage{british}%
\global\long\def\phi{\varphi}
\global\long\def\epsilon{\varepsilon}
\global\long\def\eps{\varepsilon}
\global\long\def\theta{\vartheta}
\global\long\def\E{\mathbb{E}}
\global\long\def\N{\mathbb{N}}
\global\long\def\Z{\mathbb{Z}}
\global\long\def\R{\mathbb{R}}
\global\long\def\F{\mathcal{F}}
\global\long\def\le{\leqslant}
\global\long\def\ge{\geqslant}
\global\long\def\MT{\ensuremath{\clubsuit}}
\global\long\def\1{\mathbbm1}
\global\long\def\d{\mathrm{d}}
\global\long\def\P{\mathbb{P}}
 \global\long\def\subset{\subseteq}
\global\long\def\supset{\supseteq}
\global\long\def\argmin{\arg\,\min}
\global\long\def\bull{{\scriptstyle \bullet}}
\global\long\def\supp{\operatorname{supp}}
\global\long\def\sgn{\operatorname{sign}}

\selectlanguage{english}%

\title{On infinitely divisible distributions with polynomially decaying
characteristic functions}

\author{Mathias Trabs%
\thanks{Humboldt-Universität zu Berlin, Email: trabs@math.hu-berlin.de%
}}
\maketitle
\begin{abstract}
We provide necessary and sufficient conditions on the characteristics
of an infinitely divisible distribution under which its characteristic
function $\phi$ decays polynomially. Under a mild regularity condition
this polynomial decay is equivalent to $1/\phi$ being a Fourier multiplier
on Besov spaces.
\end{abstract}

\begin{doublespace}
\end{doublespace}

\section{Introduction}

Any infinitely divisible distribution (IDD) $\mu$ is uniquely determined
by its characteristic triplet $(\sigma^{2},\gamma,\nu)$ where $\sigma^{2}>0$
is the diffusion coefficient, $\gamma\in\R$ is a drift parameter
and $\nu$ is the so-called Lévy measure. By the Lévy-Khintchine formula
the characteristic function of $\mu$ is given by 
\begin{align*}
\phi(u): & =\F\mu(u):=\int e^{iux}\mu(\d x)\\
 & =\exp\Big(-\frac{\sigma^{2}}{2}u^{2}+i\gamma u+\int\big(e^{iux}-1-iux\1_{[-1,1]}(x)\big)\nu(\d x)\Big),\quad u\in\R.
\end{align*}
The key question of this article is under which conditions on the
characteristic triplet $|\phi(u)|$ decays polynomially for $|u|\to\infty$.
Note that, first, $|\phi(u)|$ does not depend on $\gamma$ and, second,
if $\sigma^{2}>0$ then $|\phi(u)|$ is of the order $e^{-\sigma^{2}u^{2}/2}$.
Hence, we basically have to study the interplay between the Lévy measure
$\nu$ and the behavior of $|\phi(u)|$ as $|u|\to\infty$.

Due to the connection between the decay of a characteristic function
$\phi$ and the regularity of the transition density of the corresponding
Lévy processes established by \citet{hartmanWintner1942}, upper bounds
for $|\phi|$ have attracted a certain interest in the literature.
\citet{orey1968,kallenberg1981,knopovaSchilling2013} study necessary
and sufficient conditions for an exponential decay of $\phi$ and
thus for the existence of infinitely smooth transition densities.
The less regular case of polynomially decaying characteristic functions
is essentially studied for self-decomposable distributions only. Here,
the existence of polynomial upper bounds is in detail analyzed, see
\citet[Chap. 28]{sato1999} and references therein. 

While upper bounds for $|\phi|$ are more interesting from the probabilistic
perspective, lower bounds are highly relevant from a statistical point
of view. Surprisingly, polynomially lower bounds are only known for
several explicit parametric classes of IDDs, for instance, the family
of Gamma distributions. \citet{trabs:2011} has established a polynomial
lower bound for a class of Lévy processes which is slightly larger
than self-decomposable processes. 

Inspired by the results on self-decomposable distributions, we will
show that under a mild regularity assumption on $\nu$ in a neighborhood
of the origin, $\phi$ decays polynomially if and only if there is
no diffusion component and the Lebesgue density of $x\nu(\d x)$ (which
we assume to exist at zero) is bounded. From the degree of the polynomial
decay we conclude how many continuous derivatives the probability
density admits and we will show that this number is sharp in the sense
that there cannot be more derivatives, generalizing the result on
self-decomposable distributions by \citet{wolfe1971}. 

Let us illustrate the statistical interest in lower bounds on $|\phi|$
in prototypical deconvolution model. We observe an i.i.d. sample 
\begin{equation}
Y_{i}=X_{i}+\eps_{i},\quad i=1,\dots,n,\label{eq:decon}
\end{equation}
where the target population $X_{1},\dots,X_{n}$ is corrupted by independent
additive noise $\eps_{i}$. In many applications the error $\epsilon_{i}$
can be understood as an aggregation of many small independent influences.
Therefore, IDDs build a natural class of error distributions since
they can be characterized as the class of limit distributions of the
sum of independent, uniformly asymptotically negligible random variables.
Indeed, popular examples like the normal or the Laplace distribution
are IDD. As shown by \citet{fan1991} nonparametric convergence rates
for estimating the distribution of $X_{i}$ depend on the decay of
the characteristic function of $\eps_{i}$. In particular, a polynomial
decay corresponds to mildly ill-posed estimation problems allowing
polynomial convergence rates which is much faster than the logarithmic
rates in the case of an exponential decay. Due to the auto-deconvolution
structure of discretely observed Lévy processes reported by \citet{belomestnyReiss2006},
estimating the characteristic triplet of a Lévy process depends on
the decay of the characteristic function of the marginal IDD, too.

Since the observations $Y_{i}$ are distributed according to the convolution
of the distributions of $X_{i}$ and $\eps_{i}$, we have to divide
in the Fourier domain the (estimated) characteristic function of $Y_{i}$
by the characteristic function of $\eps_{i}$ to assess the distribution
of $X_{i}$. This spectral approach gives raise to the map $f\mapsto\F^{-1}[\F f/\phi]$.
Slightly abusing notation, we consequently denote $\F^{-1}[1/\phi]$
as the \emph{deconvolution operator} which has a prominent role in
the statistical analysis of the deconvolution model and related models.
To analyze its mapping properties, the Fourier multiplier approach
by \citet{nicklReiss2012} is extremely useful. Studying a Lévy process
model, they have shown under certain sufficient assumptions that $1/\phi$
is a Fourier multiplier on Besov spaces. We refer to \citet{triebel2010}
for definitions and properties of the Besov spaces $B_{p,q}^{s}(\R),s\in\R,p,q\in[1,\infty]$.
We say $1/\phi$ satisfies the \emph{Fourier multiplier property}
if there exists some $\alpha>0$ such that for all $s\in\R,1\le p,q\le\infty$
the linear map 
\begin{equation}
B_{p,q}^{s+\alpha}(\R)\ni f\mapsto\F^{-1}\Big[\frac{\F f}{\phi}\Big]\in B_{p,q}^{s}(\R)\label{eq:fmp}
\end{equation}
is bounded. In the deconvolution model the Fourier multiplier approach
and the closely related pseudo-differential operators were exploited
in a series of recent papers \citet{soehlTrabs2012,SchmidtHieberEtAll2012,dattnerEtAl2013}
as well as \citet{Diss2014} for an overview. Minimal conditions on
the IDD which imply the Fourier multiplier property are desirable
to be able to apply this approach in a wide area of models. It turns
out that this mapping property is very natural in the context of IDDs:
We show that a polynomial decay of the characteristic function is
necessary and (under a mild regularity condition) sufficient to conclude
that $1/\phi$ is a Fourier multiplier on Besov spaces.

\section{Polynomial decay of the characteristic function}

Before we precisely state our results, let us introduce some notation.
The space of finite signed Borel measures on the real line will be
denoted by $\mathcal{M}(\R).$ For any $\mu\in\mathcal{M}(\R)$ there
are two positive finite measures $\mu^{+},\mu^{-}$ such that $\mu(A)=\mu^{+}(A)-\mu^{-}(A)$
for any $A\in\mathscr{B(\R).}$ Using this so-called Jordan-decomposition,
the total variation norm on $\mathcal{M}(\R)$ is defined by 
\[
\|\mu\|_{TV}:=\mu^{+}(\R)+\mu^{-}(\R).
\]
If a function $f\colon\R\to\R$ is locally integrable, it defines
a distribution $T_{f}(\psi)=\int\psi f$ on the test function space
$\mathcal{D}(\R)$ of infinitely smooth functions with bounded support.
If the distributional derivative of $T_{f}$ is a finite signed measure
$\rho\in\mathcal{M}(\R)$, the function $f$ is called (weakly) differentiable
with derivative $Df:=\rho.$ The space of functions of bounded variation
is then defined by
\[
BV(\R):=\left\{ f\colon\R\to\R\text{ locally integrable},Df\in\mathcal{M}(\R)\right\} 
\]
with bounded variation norm $\|f\|_{BV}:=\|Df\|_{TV}$ for $f\in BV(\R)$.
The measure $Df$ satisfies $Df((a,b])=f(b+)-f(a+)$ for $-\infty<a<b<\infty.$
We will write $\lesssim,\gtrsim$ for inequalities up to constants.

Let $\mu$ be an IDD with characteristic triplet $(\sigma^{2},\gamma,\nu)$.
Defining the symmetrized jump measure by 
\[
\nu_{s}(A):=\nu(A)+\nu(-A)
\]
 for any Borel set $A\in\mathscr{B}(\R_{+}),$ the absolute value
of the characteristic function is given by 
\begin{align*}
|\phi(u)| & =\exp\Big(-\frac{\sigma^{2}u^{2}}{2}+\int(\cos(ux)-1)\nu(\d x)\Big)\\
 & =\exp\Big(-\frac{\sigma^{2}u^{2}}{2}+\int_{0}^{\infty}(\cos(ux)-1)\nu_{s}(\d x)\Big).
\end{align*}

The following proposition is inspired by the behavior of self-decomposable
distributions. To prove it, we shall generalize Lemma 2.1 by \citet{trabs:2011}
and its counterpart Lemma 53.9 in \citet{sato1999}. It turns out
that a polynomial decay of the characteristic function holds true
for a class of IDDs which is much larger than self-decomposable distributions. 
\begin{prop}
\label{prop:PolDecay} If $\sigma^{2}=0$ and $x\nu_{s}(\d x)$ admits
on an interval $[0,\delta],$ for some $\delta>0,$ a Lebesgue density
$k_{s}\in BV([0,\delta])$, then for any $\eps>0$ 
\[
(1+|u|)^{-\alpha-\eps}\lesssim|\phi(u)|\lesssim(1+|u|)^{-\alpha+\eps}\qquad\text{with\quad}\alpha:=k_{s}(0+).
\]
 If, moreover, $\int_{0}^{\delta}\log(y^{-1})(Dk_{s})^{+}(\d y)<\infty$,
then $|\phi(u)|\gtrsim(1+|u|)^{-\alpha}$. \end{prop}
\begin{proof}
Using 
\[
|\phi(u)|=\exp\Big(\int_{0}^{\infty}\frac{\cos(ux)-1}{x}k_{s}(x)\d x\Big)
\]
 and the symmetry of the cosine function, we can assume $u>0$ without
loss of generality. We denote $\tilde{\rho}=Dk_{s}$ and let $\tau\in(0,\delta\wedge1)$.
Since $\|k_{s}\1_{[0,\delta]}\|_{\infty}\le\|k_{s}\1_{[0,\delta]}\|_{BV}<\infty,$
we estimate for $u\le1/\tau$ 
\begin{align*}
1\ge|\phi(u)|= & \exp\Big(\int_{0}^{\delta}\frac{\cos(ux)-1}{x}k_{s}(x)\d x+\int_{\delta}^{\infty}(\cos(ux)-1)\nu_{s}(\d x)\Big)\\
\ge & \exp\Big(\|k_{s}\|_{L^{\infty}([0,\delta])}\int_{0}^{u\delta}\frac{\cos x-1}{x}\d x-2\int_{\delta}^{\infty}\nu_{s}(\d x)\Big)\\
\ge & \exp\Big(-\|k_{s}\|_{L^{\infty}([0,\delta])}\sup_{v\in(0,\delta/\tau]}\int_{0}^{v}\frac{1-\cos x}{x}\d x-2\int_{\delta}^{\infty}\nu_{s}(\d x)\Big),
\end{align*}
 where the last line is a positive constant independent of $u$. It
remains to consider the tail behavior of $|\phi(u)|$ for $u>1/\tau$.

To show the upper bound of $|\phi(u)|$ for $u>1/\tau$. We use Fubini's
theorem and the finite constant $c_{1}:=\sup_{v\ge1}\int_{1}^{v}\frac{\cos x}{x}\d x>0$
to estimate 
\begin{align}
|\phi(u)|\le & \exp\Big(\int_{1/u}^{\tau}\frac{\cos(ux)-1}{x}k_{s}(x)\d x\Big)\nonumber \\
= & \exp\Big(\int_{1/u}^{\tau}\frac{\cos(ux)-1}{x}\int_{1/u}^{x}\tilde{\rho}(\d y)\d x+k_{s}(\tfrac{1}{u}+)\int_{1}^{\tau u}\frac{\cos(x)-1}{x}\d x\Big)\nonumber \\
= & \exp\Big(\int_{1/u}^{\tau}\int_{y}^{\tau}\frac{\cos(ux)-1}{x}\d x\tilde{\rho}(\d y)+k_{s}(\tfrac{1}{u}+)\Big(\int_{1}^{\tau u}\frac{\cos x}{x}\d x-\log(\tau u)\Big)\Big)\nonumber \\
\le & \exp\Big(2\int_{1/u}^{\tau}\big(\log\tau+\log(\tfrac{1}{y})\big)\tilde{\rho}^{-}(\d y)-(\log u)k_{s}(\tfrac{1}{u}+)\nonumber \\
 & \qquad\qquad+\bigl(c_{1}+\log(\tfrac{1}{\tau})\bigr)k_{s}\big(\tfrac{1}{u}+\big)\Big)\nonumber \\
\le & \exp\Big(-(\log u)\big(k_{s}(\tfrac{1}{u}+)-2\int_{1/u}^{\tau}\tilde{\rho}^{-}(\d y)\big)+(c_{1}+\log(\tfrac{1}{\tau}))k_{s}\big(\tfrac{1}{u}+\big)\Big).\label{eqPhiUpperBound}
\end{align}
 Hence, for any $\epsilon>0$ we find some $\tau\in(0,\delta)$ which
is sufficiently small such that $\left|\phi(u)\right|\lesssim u^{-(\alpha-\epsilon)}$
for all $u>1/\tau.$

To verify the lower bound, we decompose for $\tau\in(0,\delta)$ and
$u>1/\tau$ 
\begin{align}
|\phi(u)|= & \exp\Big(\Big(\int_{0}^{1/u}+\int_{1/u}^{\tau}\Big)\frac{\cos(ux)-1}{x}k_{s}(x)\d x+\int_{\tau}^{\infty}(\cos(ux)-1)\nu_{s}(\d x)\Big),\label{eqPhiDecomp}
\end{align}
 where the three integrals will be bounded separately from below.
Using $\|k_{s}\1_{[0,\delta]}\|_{\infty}<\infty$, we estimate the
first integral by 
\[
\int_{0}^{1/u}\frac{\cos(ux)-1}{x}k_{s}(x)\d x\ge\|k_{s}\1_{[0,\delta]}\|_{\infty}\int_{0}^{1}\frac{\cos x-1}{x}\d x,
\]
 where the integral on the right-hand side is a negative finite constant
independent of $u$. The third integral in (\ref{eqPhiDecomp}) can
be bounded by 
\[
\int_{\tau}^{\infty}(\cos(ux)-1)\nu_{s}(\d x)\ge-2\int_{\tau}^{\infty}\nu_{s}(\d x).
\]
 It remains to estimate the second integral, where we proceed similarly
to the upper bound. We obtain with the finite constant $c_{2}:=\inf_{v\ge1}\int_{1}^{v}\frac{\cos x}{x}\d x<0$
\begin{align*}
 & \int_{1/u}^{\tau}\frac{\cos(ux)-1}{x}k_{s}(x)\d x\\
 & \qquad=\int_{1/u}^{\tau}\int_{y}^{\tau}\frac{\cos(ux)-1}{x}\d x\tilde{\rho}(\d y)+k_{s}(\tfrac{1}{u}+)\big(\int_{1}^{\tau u}\frac{\cos x}{x}\d x-\log(\tau u)\big)\\
 & \qquad\ge-2\int_{1/u}^{\tau}\big(\log\tau+\log(y^{-1})\big)\tilde{\rho}^{+}(\d y)-k_{s}(\tfrac{1}{u}+)\log u+c_{2}k_{s}\big(\tfrac{1}{u}+\big)\\
 & \qquad=-(\log u)\Big(k_{s}(\tfrac{1}{u}+)+2\int_{1/u}^{\tau}\tilde{\rho}^{+}(\d y)\Big)+c_{2}k_{s}\big(\tfrac{1}{u}+\big).
\end{align*}
 This yields $\left|\phi(u)\right|\gtrsim u^{-(\alpha+\epsilon)}$
for any $\epsilon>0$ and for all $u>1/\tau$ with $\tau$ sufficiently
small.

The addendum follows from the previous estimates, the additional assumption
and 
\begin{align*}
 & \int_{1/u}^{\tau}\frac{\cos(ux)-1}{x}k_{s}(x)\d x\\
 & \qquad=\int_{0}^{\tau}\int_{y}^{\tau}\frac{\cos(ux)-1}{x}\d x\tilde{\rho}(\d y)+k_{s}(0+)\Big(\int_{1}^{\tau u}\frac{\cos x}{x}\d x-\log(\tau u)\Big)\\
 & \qquad\ge-2\int_{0}^{\tau}\log(y^{-1})\tilde{\rho}^{+}(\d y)-\alpha\log u+c_{2}\alpha.
\end{align*}
 \end{proof}
\begin{example*}
If $\mu$ is a self-decomposable distribution, then the Lévy measure
admits always a Lebesgue density of the form $\nu(\d x)=\frac{k(x)}{|x|}\d x$
where the so-called k-function $k$ is increasing on the negative
half line and decreasing on the positive half line. Hence, $k_{s}\1_{(0,\delta]}\in BV(\R)$
is equivalent to $\|k\|_{\infty}<\infty$. The condition $\int_{0}^{\delta}\log(y^{-1})(Dk_{s})^{+}(\d y)<\infty$
is always satisfied for self-decomposable distributions owing to $(Dk_{s})^{+}=0$.
As indicated by \citet{trabs:2011} it is more generally sufficient
if the quotient $(k_{s}(x+)-\alpha)/x$ is bounded from above uniformly
in $x\in(0,\eps]$ for some $\eps>0$, meaning that the largest slope
of $k_{s}$ near zero is bounded, to obtain the sharp bound of the
decay rate. 
\end{example*}
With Proposition~\pageref{prop:PolDecay} at hand we can characterize
IDDs with polynomially decaying characteristic functions under the
following regularity assumption on $\nu$ near the origin. 
\begin{assumption}
\label{ass:Regularity}Assume that $x\nu_{s}(\d x)$ admits on a small
interval $(0,\delta],$ for some $\delta\in(0,1),$ a Lebesgue density
$k_{s}$ with $k_{s}\in BV([\eps,\delta])$ for all $\eps\in(0,\delta)$.
Suppose that $\|(Dk_{s})^{+}\|_{TV([0,\delta])}=\lim_{\eps\downarrow0}\|(Dk_{s})^{+}\|_{TV([\eps,\delta])}<\infty.$ 
\end{assumption}
Assumption~\ref{ass:Regularity} basically excludes Lévy measures
which oscillate at zero or have additional singularities in any neighborhood
of the origin. Both possibilities are not natural in applications,
for instance, for the deconvolution setting or the modeling of stochastic
processes via Lévy processes.
\begin{thm}
\label{thm:CharPolDec} Let $\mu$ be an infinitely divisible distribution
satisfying Assumption~\ref{ass:Regularity}. Then the following are
equivalent: 
\begin{enumerate}
\item $\sigma^{2}=0$ and $\|k_{s}\|_{BV([0,\delta])}<\infty$, 
\item $\sigma^{2}=0$ and $\lim_{x\to0}k_{s}(x)<\infty,$ 
\item there is some $\alpha>0$ and some constant $c>0$ such that $|\phi(u)|\ge c(1+|u|)^{-\alpha}$
for all $u\in\R$. 
\end{enumerate}
\end{thm}
\begin{proof}
Denoting $\tilde{\rho}:=Dk_{s},$ we define the monotone decreasing
functions $k_{s}^{+}(x):=\int_{x}^{\delta}\tilde{\rho}^{+}(\d y)$
and $k_{s}^{-}(x):=\int_{x}^{\delta}\tilde{\rho}^{-}(\d y)$ for $x\in(0,\delta]$.
To verify equivalence of (i) and (ii), we conclude from $k_{s}(\delta+)-k_{s}(x+)=k_{s}^{+}(x)-k_{s}^{-}(x)$
that 
\begin{align*}
\|k_{s}\|_{BV([0,\delta])}-\|\tilde{\rho}^{+}\|_{TV([0,\delta])} & =\|\tilde{\rho}^{-}\|_{TV([0,\delta])}\\
 & =\sup_{x\in(0,\delta]}k_{s}^{-}(x)=\sup_{x\in(0,\delta]}\big\{ k_{s}(x+)+k_{s}^{+}(x)-k_{s}(\delta+)\big\}.
\end{align*}

Hence, $\|k_{s}\|_{BV([0,\delta])}<\infty$ if and only if $k_{s}(0+)<\infty$,
owing to $0\le k_{s}^{+}(0+)=\|\tilde{\rho}^{+}\|_{TV([0,\delta])}.$

The inclusion (i)$\Rightarrow$(iii) immediately follows from the
lower bound in Proposition~\ref{prop:PolDecay}.

To show (iii)$\Rightarrow$(i), we first note that if $\sigma^{2}>0$
then $|\phi(u)|\lesssim e^{-cu^{2}},u\in\R,$ for some constant $c>0$
which contradicts (iii). So let $\sigma^{2}=0$ and $u>0$ without
loss of generality. We deduce similarly to (\ref{eqPhiUpperBound})
and for any $\tau\in(0,\delta)$ and any $u>1/\tau$ 
\begin{align}
|\phi(u)|\le & \exp\Big(\int_{1/u}^{\tau}\int_{y}^{\tau}\frac{1-\cos(ux)}{x}\d x\tilde{\rho}^{-}(\d y)+k_{s}(\tfrac{1}{u}+)\int_{1}^{\tau u}\frac{\cos(x)-1}{x}\d x\Big)\nonumber \\
\le & \exp\Big(-\int_{1}^{\tau u}\frac{1-\cos(x)}{x}\d x\Big(k_{s}(\tfrac{1}{u}+)-\int_{1/u}^{\tau}\tilde{\rho}^{-}(\d y)\Big)\nonumber \\
= & \exp\Big(-\int_{1}^{\tau u}\frac{1-\cos(x)}{x}\d x\big(k_{s}(\delta+)-k_{s}^{+}(\tfrac{1}{u})+k_{s}^{-}(\tau)\big)\Big).\label{eq:phiUpB}
\end{align}
 Using 
\[
\int_{1}^{y}\frac{1-\cos(x)}{x}\d x=\log y+\int_{y}^{\infty}\frac{\cos x}{x}\d x+c_{1}\quad\text{with \quad}c_{1}:=-\int_{1}^{\infty}\frac{\cos x}{x}\d x>0
\]
 and $\lim_{y\to\infty}\int_{y}^{\infty}\frac{\cos x}{x}\d x=0$,
we find a function $f\colon[1,\infty)\to\R$ such that for some $T>1$
\[
\int_{1}^{y}\frac{1-\cos(x)}{x}\d x=f(y)\log y\quad\text{and }\quad f(y)\in(\tfrac{1}{2},2)\text{ for all }y\ge T.
\]
 Combining the lower bound on $|\phi(u)|$ in Condition (iii) and
the upper bound (\ref{eq:phiUpB}), we conclude 
\[
\log c-\alpha\log2-\alpha\log u\le\log|\phi(u)|\le-f(\tau u)\log(\tau u)\big(k_{s}(\delta+)-k_{s}^{+}(\tfrac{1}{u})+k_{s}^{-}(\tau)\big)
\]
 and thus for $c_{2}:=\alpha\log2-\log c$ 
\begin{align*}
 & \log u\big(f(\tau u)\big(k_{s}(\delta+)-k_{s}^{+}(\tfrac{1}{u})+k_{s}^{-}(\tau)\big)-\alpha\big)\\
 & \qquad\le c_{2}+\log(\tau^{-1})f(\tau u)\big(k_{s}(\delta+)-k_{s}^{+}(\tfrac{1}{u})+k_{s}^{-}(\tau)\big).
\end{align*}
 That implies for $u\ge T/\tau>1$ 
\[
\log u\big(\tfrac{1}{2}k_{s}(\delta+)-2\|\tilde{\rho}^{+}\|_{TV([0,\delta])}+\tfrac{1}{2}k_{s}^{-}(\tau)\big)-\alpha\big)\le c_{2}+2\log(\tau^{-1})\big(k_{s}(\delta+)+k_{s}^{-}(\tau)\big).
\]
 Since the right-hand side is independent of $u$ and $\log u\to\infty$
as $u\to\infty$, we obtain $\tfrac{1}{2}k_{s}(\delta+)-2\|\tilde{\rho}^{+}\|_{TV([0,\delta])}+\tfrac{1}{2}k_{s}^{-}(\tau)-\alpha<0$
and therefore 
\[
\|\tilde{\rho}^{-}\|_{TV([0,\delta])}=\sup_{\tau\in(0,\delta)}k_{s}^{-}(\tau)\le2\alpha+4\|\tilde{\rho}^{+}\|_{TV([0,\delta])}<\infty,
\]
 implying $\|k_{s}\|_{BV([0,\delta])}<\infty.$
\end{proof}
We conclude the following corollary on the regularity of the probability
densities of IDDs. Later we will show that there cannot exist more
than $\alpha-1$ derivatives.
\begin{cor}
\label{cor:densities}Let $\mu$ be an infinitely divisible distribution
satisfying Assumption~\ref{ass:Regularity} with $k_{s}\in BV([0,\delta])$
and suppose that $x\nu$ is absolutely continuous on $\R.$ Let $\alpha:=k_{s}(0+)>1$.
Then $\mu$ admits a Lebesgue-density $f\in C^{n}(\R)$ for any integer
$0\le n<\alpha-1$.\end{cor}
\begin{proof}
Due to the absolute continuity of $\nu$ and infinite activity $\int_{\R}\nu(\d x)=\infty$
by $\alpha>0$, the measure $\mu$ is absolutely continuous \citep[Thm. 27.7]{sato1999}.
Proposition~\ref{prop:PolDecay} yield $|\phi(u)|\lesssim(1+|u|)^{-\alpha+\eps}$
for any $\eps>0$. Therefore, $\int_{\R}|\phi(u)||u|^{n}\d u<\infty$
for any $n<\alpha-1$. By standard Fourier analysis we obtain $f:=\F^{-1}\phi\in C^{n}(\R).$
\end{proof}

\section{A Fourier multiplier theorem\label{sec:fourierMultplier}}

To motivate Fourier multiplier theorem and to explain its statistical
relevance, let us dive a little deeper into the statistical analysis
of the deconvolution model~\eqref{eq:decon}. Assume the laws of
$X_{1}$ and $\eps_{1}$ have Lebesgue densities $f$ and $f_{\eps}$,
respectively. Consequently, $Y_{1}$ is distributed according to $f_{Y}=f\ast f_{\eps}$.
Denoting the characteristic functions of $Y_{1}$ and $\eps_{1}$
by $\phi_{Y}$ and $\phi_{\eps}$, respectively, we obtain the Fourier
inversion formula 
\[
f=\F^{-1}\Big[\frac{\phi_{Y}}{\phi_{\eps}}\Big]=\F^{-1}\Big[\frac{1}{\phi_{\eps}}\Big]\ast f_{Y},
\]
where the second equality has to be understood in distributional sense.
Hence, the deconvolution map is given by $g\mapsto\F^{-1}[\F g/\phi_{\eps}]$.
Slightly abusing notation, we denote it by $\F^{-1}[1/\phi_{\eps}]$.
Replacing $\phi_{Y}$ by the empirical characteristic function of
$(Y_{j})$ and regularizing with a band-limited kernel $K$ with bandwidth
$h>0$, the previous display gives us immediately the natural kernel
density estimator, which was proposed by \citet{fan1991} 
\[
\hat{f}_{h}:=\F^{-1}\Big[\frac{\F K(h\bull)}{\phi_{\eps}}\Big]\ast\mu_{n}
\]
 for the empirical measure $\mu_{n}=\sum_{j=1}^{n}\delta_{Y_{j}}$
with Dirac measure $\delta_{y}$ in the point $y\in\R$. To estimate
the linear functional $\int\zeta(x)f(x)\d x$ for suitable functions
$\zeta$, say $\zeta\in L^{1}(\R)\cap L^{\infty}(\R)$, a plug-in
approach yields 
\begin{align*}
\int\zeta(x)\hat{f}_{h}(x)\d x & =\int\zeta(x)\Big(\F^{-1}\Big[\frac{\F K(h\bull)}{\phi_{\eps}}\Big]\ast\mu_{n}\Big)(x)\d x\\
 & =\int\Big(\F^{-1}\Big[\frac{\F K(-h\bull)}{\phi_{\eps}(-\bull)}\Big]\ast\zeta\Big)(x)\mu_{n}(\d x).
\end{align*}
Since the regularizing $\F K(h\bull)$ degenerates to one as $h\downarrow0$,
we have to study the mapping properties of the deconvolution operator
$\F^{-1}[1/\phi_{\eps}(-\bull)]$. 

Applying a Mihlin multiplier theorem, \citet{nicklReiss2012} quantified
the mapping properties of the Fourier multiplier in a Besov scale
under certain assumptions. However, their conditions already exclude
simple examples like the Poisson process, which satisfies \eqref{eq:fmp}
even for $\alpha=0$. The target of this section is to study necessary
and sufficient conditions under which $1/\phi$ satisfies the Fourier
multiplier property. 

Using that $B_{p,q}^{s}(\R)\ni f\mapsto\F^{-1}[(1+iu)^{\alpha}\F f]\in B_{p,q}^{s+\alpha}(\R)$
is an isomorphism for any $\alpha\in\R$ and that the set of Fourier
multipliers on $B_{p,q}^{s}(\R)$ does not depend on $s,q$ and are
nested in $p$ \citep[Prop. 2.6.2 and Thm. 2.6.3]{triebel2010}, the
characteristic function $\phi$ has to satisfy necessarily 
\[
|\phi(u)|\gtrsim(1+|u|)^{-\alpha},\quad u\in\R.
\]
 In the previous section we have characterized infinitely divisible
distributions whose characteristic function decays only with a polynomial
rate. Our main result shows that under the regularity Assumption~\ref{ass:Regularity}
the necessary polynomial decay of $\phi$ already implies that $1/\phi$
is a Fourier multiplier on Besov spaces.
\begin{thm}
\label{thm:FourierMult} Let $\mu$ be an infinitely divisible distribution
with characteristic function $\phi$ satisfying Assumption~\ref{ass:Regularity}.
If and only if one (and thus all) of the conditions (i) to (iii) of
Theorem~\ref{thm:CharPolDec} are satisfied, $1/\phi$ is a Fourier
multiplier on Besov spaces: There exists some $\alpha>0$ such that
for all $s\in\R,1\le p,q\le\infty$ the linear map 
\[
B_{p,q}^{s+\alpha}(\R)\ni f\mapsto\F^{-1}\Big[\frac{\F f}{\phi}\Big]\in B_{p,q}^{s}(\R)
\]
 is bounded.\end{thm}
\begin{proof}
If $1/\phi$ is a Fourier multiplier, then (iii) in Theorem~\ref{thm:CharPolDec}
has to be fulfilled as carried out above.

Now, let the assumptions of Theorem~\ref{thm:CharPolDec} be satisfied.
Using $\sigma^{2}=0$ and noting that $k_{s}\in BV([0,\delta])$ implies
boundedness of $k_{s}$ and thus $\int_{0}^{1}x\nu_{s}(\d x)<\infty$,
the characteristic function of $\mu$ can be represented by 
\[
\phi(u)=\exp\Big(i\gamma_{0}u+\int\big(e^{iux}-1\big)\nu(\d x)\Big)=\phi_{c}(u)\phi_{p}(u),\quad\text{for }\quad u\in\R,
\]
 where $\gamma_{0}=\gamma-\int_{0}^{1}x\nu_{s}(\d x)\in\R$ is a drift
parameter and 
\[
\phi_{c}(u):=\exp\Big(\int_{[-\delta,\delta]}\big(e^{iux}-1\big)\nu(\d x)\Big),\quad\phi_{p}(u):=\exp\Big(i\gamma_{0}u+\int_{\R\setminus[-\delta,\delta]}\big(e^{iux}-1\big)\nu(\d x)\Big).
\]
 Defining $\mu_{c}:=\F^{-1}[\phi_{c}]$ and $\mu_{c}:=\F^{-1}[\phi_{p}]$,
this yields the decomposition $\mu=\mu_{c}\ast\mu_{p}$ into a convolution
of an IDD with compactly supported jump measure and distribution of
compound Poisson type. The deconvolution operator decomposes into
a composition $\F^{-1}[1/\phi]=\F^{-1}[1/\phi^{c}]\ast\F^{-1}[1/\phi^{p}]$,
where 
\[
\F^{-1}[1/\phi_{p}]=\delta_{-\gamma_{0}}\ast\Big(e^{\nu(\R\setminus[-\delta,\delta])}\sum_{k=0}^{\infty}\frac{(-1)^{k}}{k!}\big(\nu|_{(\R\setminus[-\delta,\delta]}\big)^{\ast k}\Big)
\]
 is a finite signed measure. Since convolution with a finite signed
measure is a bounded map on $L^{p}(\R)$ for any $0<p\le\infty,$
we conclude from the Littlewood--Paley representation of Besov spaces
that $B_{p,q}^{s}(\R)\ni f\mapsto\F^{-1}[1/\phi_{p}]\ast f\in B_{p,q}^{s}(\R)$
is a bounded linear map. For $\phi_{c}$ we use $\phi_{c}'(u)=i\F[x\nu|_{[-\delta,\delta]}](u)\phi_{c}(u)$
and and the estimate $|\F[x\nu|_{[-\delta,\delta]}](u)|\lesssim(1+|u|)^{-1}$
by the bounded variation of $x\nu$ near the origin. The polynomial
decay of $\phi_{c}$ implies then for some $\alpha>0$ 
\begin{align*}
(1+|u|)^{-\alpha}|\phi_{c}^{-1}(u)| & \lesssim1\quad\text{and}\\
\quad(1+|u|)^{-\alpha}|(\phi_{c}^{-1})'(u)| & =(1+|u|)^{-\alpha}|\F[x\nu|_{[-\delta,\delta]}](u)/\phi_{c}(u)|\lesssim(1+|u|)^{-1}.
\end{align*}
 Therefore, we can apply Corollary~4.11 by \citet{girardiWeis2003}
to conclude that $(1+|u|)^{-\alpha}/\phi_{c}(u)$ is a Fourier muliplier
on all Besov spaces $B_{p,q}^{s}(\R)$ for all $s\in\R,1\le p,q\le\infty.$
The assertion follows because $f\mapsto\F^{-1}[(1+iu)^{\alpha}\F f]$
is an isomorphism from $B_{p,q}^{s+\alpha}(\R)$ onto $B_{p,q}^{s}(\R)$.
\end{proof}
As a small application of Theorem~\ref{thm:FourierMult} we will
extend Corollary~\ref{cor:densities} generalizing of the result
for self-decomposable distributions by \citet{wolfe1971} to a much
larger class, but for simplicity restricted to non-integer $\alpha=k_{s}(0+).$
\begin{cor}
Let $\mu$ be an infinitely divisible distribution satisfying Assumption~\ref{ass:Regularity}
with $k_{s}\in BV([0,\delta])$ and suppose that $x\nu$ is absolutely
continuous on $\R.$ Let $\alpha:=k_{s}(0+)\notin\mathbb{N}$ be positive.
Then the Lebesgue-density of $\mu$ satisfies $f\in C^{n}(\R)$ for
$n\in\mathbb{N}$ if and only if $0\le n<\alpha-1$.\end{cor}
\begin{proof}
Let $n$ satisfy $n-1<\alpha<n+2$ and suppose $f\in C^{n+1}(\R)$.
For the function $g(x)=e^{-x}\1_{[0,\infty)},x\in\R,$ a direct calculation
or using $g\in B_{1,1}^{1}(\R)$ shows $f\ast g\in C^{n+2}(\R)\subset B_{\infty,\infty}^{n+2}$.
Then Theorem~\ref{thm:FourierMult} applied to $f\ast g$ yields
for any $\epsilon>0$ that $g=\F^{-1}[\F[f\ast g]/\phi]\in B_{\infty,\infty}^{n+2-\alpha-\eps}(\R)$.
Since $n+2-\alpha>0$ and $\epsilon$ can be chosen small enough,
the Besov embedding yields $g\in B_{\infty,\infty}^{n+2-\alpha-\eps}(\R)\subset C^{0}(\R)$
which contradicts to $g$ being discontinuous.
\end{proof}
\bibliographystyle{chicago}
\bibliography{library}

\end{document}